%% file: IRO23.tex
\documentclass[12pt,oneside,openany,article]{memoir}
\usepackage{mempatch}

\nouppercaseheads 
\usepackage[amsthm,thmmarks,hyperref]{ntheorem}
\usepackage[noTeX]{mmap}
\usepackage{etex}

\usepackage[english]{babel}
\usepackage[leqno]{mathtools}

\usepackage{mathrsfs}
\usepackage{upgreek}
\usepackage[compress,square,comma,numbers]{natbib}
\usepackage{hypernat} 

\usepackage{sfmath}

\usepackage{microtype}

\usepackage{subfig}
\usepackage{wrapfig}
\usepackage{array}

\usepackage{graphicx,color}
\definecolor{gray}{gray}{0}
\pagecolor{white}

\usepackage{hyperxmp}
\usepackage{xr-hyper}
\usepackage{nameref}
\usepackage[pdftex,bookmarks,pdfnewwindow,plainpages=false,unicode]{hyperref}

\usepackage{bookmark}

\usepackage{enumitem}

\usepackage{amssymb} 

\hypersetup{
colorlinks=true,
linkcolor=black,
citecolor=black,
urlcolor=blue,
pdfauthor={Victor Ivrii},
pdftitle={Short Loops and Pointwise Spectral Asymptotics},
pdfsubject={Sharp Spectral Asymptotics},
pdfkeywords={Microlocal Analysis,  Sharp Spectral Asymptotics, Schort loops},
bookmarksdepth={4}
}

\numberwithin{equation}{chapter}

\input{defines_IRO23}

\begin{document}
\title{Short Loops and Pointwise Spectral Asymptotics}
\author{Victor Ivrii}

\maketitle
{\abstract%
We consider pointwise semiclassical spectral asymptotics i.e. asymptotics of $e(x,x,0)$ as $h\to +0$ where $e(x,y,\tau)$ is the Schwartz kernel of the spectral projector and consider two cases when schort loops give contribution above $O(h^{1-d})$:

\begin{enumerate}[label=(\roman*), leftmargin=*]
\item Schr\"odinger operator in dimensions $1,2$ as potential $V=0\implies \nabla V\ne 0$;

\item Operators near boundaries.
\end{enumerate}
\endabstract}

\setcounter{chapter}{-1}
\chapter{Introduction}
\label{sect-5-0}

We consider pointwise semiclassical spectral asymptotics i.e. asymptotics of $e(x,x,0)$ as $h\to +0$ where $e(x,y,\tau)$ is the Schwartz kernel of the spectral projector and consider two cases when schort loops give contribution above $O(h^{1-d})$:

\begin{enumerate}[label=(\roman*), leftmargin=*]
\item Schr\"odinger operator in dimensions $1,2$ as potential $V=0\implies \nabla V\ne 0$;

\item Operators near boundaries.
\end{enumerate}

This article is a rather small part of the huge project to write a book and is just part of subsection~\ref{book_new-sect-5-2-1}, appendix~\ref{book_new-sect-5-A-1} and  section~\ref{book_new-sect-8-1} of V.~Ivrii \cite{futurebook} consisting entirely of newly researched results.

\chapter{Schr\"{o}dinger operator}%
\label{sect-5-2-1}

\section{Main assumptions}
\label{sect-5-2-1-1}
Consider  Schr\"{o}dinger operator 
\begin{equation}
A=\sum_{j, k}P_jg^{jk}P_k+V, \qquad P_j=hD_j-V_j
\label{5-2-1}
\end{equation}
assuming that
\begin{claim}\label{5-2-3}
$A$ is self-adjoint operator in $\sL^2(X, \bH )$, $\bH = \bC^D$,
$D=1$ for the Schr\"odinger operator,
\end{claim}
\begin{equation}
B(0, 1)\subset X,
\label{5-2-4}
\end{equation}
\begin{phantomequation}\label{5-2-5}\end{phantomequation}
\vglue-25pt
\begin{equation}
|D^\alpha g^{jk}|\le c, \quad |D^\alpha V_j|\le c, \quad
|D^\alpha V|\le c \quad \forall \alpha : |\alpha |\le K,
\tag*{$\textup{(\ref{5-2-5})}_{1-3}$}
\end{equation}
\begin{equation}
\epsilon _0 \le \sum_{j, k}g^{jk}(x)\xi_j\xi_k |\xi |^{-2} \le c \quad
\forall x \in B(0, 1)\quad \forall \xi \in \bR^d\setminus 0.
\label{5-2-6}
\end{equation}

\begin{remark}\label{rem-5-2-1}
A (unitary) \emph{gauge transformation}%
\index{gauge transformation}%
\index{transformation!gauge}
$u \to e^{ih^{-1}\varphi (x)}u$ with a real-valued function $\varphi $ can be applied which for the Schr\"odinger and Dirac operators is equivalent to the transformation $V_j \mapsto V_j+\partial_j\varphi$. Hence one can weaken
condition $\textup{(\ref{5-2-5})}_2$, replacing it with
\begin{equation}
|D^\alpha F_{jk}|\le c \quad \forall \alpha : |\alpha | \le K
\tag*{$\textup{(\ref*{5-2-5})}'_2$}
\end{equation}
where
\begin{equation}
F_{jk}=\partial_kV_j-\partial_jV_k
\label{5-2-7}
\end{equation}
are \emph{components of the tensor magnetic intensity}.
\end{remark}

Here the principal symbol of $A$ equals
\begin{equation}
a(x, \xi )=\sum_{j, k} g^{jk}(\xi_j-V_j)(\xi_k-V_k)+V
\label{5-2-8}
\end{equation}
and the subprincipal symbol vanishes and hence the coefficients in the
spectral asymptotics are
\begin{equation}
\kappa_0(x, \tau )=(2\pi )^{-d}\varpi_d(\tau - V)_+^{d/2}\sqrt g,
\label{5-2-9}
\end{equation}
\begin{equation}
\kappa_1(x, \tau )=0
\label{5-2-10}
\end{equation}
where $\tau_1=-\infty $, $\tau_2=\tau, $ $\varpi_d$ is the volume of
the unit ball in $\bR^d$ and $g = \det (g^{jk}) ^{-1}$. We prefer to use notation $\kappa_n(x, \tau)$ rather than $\kappa_{xn}(\tau)$.

\section{General theory}
\label{sect-5-2-1-3}

Let us consider spectral asymptotics without spatial mollification. First of all, let us check $\xi$-microhyperbolicity condition; it is \begin{equation}
|V|\ge \epsilon _0 \qquad \forall x \in B(0, 1)
\label{5-2-26}
\end{equation}
and we immediately arrive to

\begin{theorem}\label{thm-5-2-5}
Let conditions \textup{(\ref{5-2-1})}, \textup{(\ref{5-2-3})}, \textup{(\ref{5-2-4})}, \textup{(\ref{5-2-5})}, \textup{(\ref{5-2-6})} and \textup{(\ref{5-2-26})} be fulfilled. Then
\begin{equation}
|\R^\W_{x, \varphi, L}|
\le C_0h^{1-d}\vartheta \bigl(\frac{h}{L}\bigr)\qquad
\forall x \in B(0, {\frac{1}{2}}).
\label{5-2-27}
\end{equation}
\end{theorem}

Now let us use the same rescaling scheme as in the alternative proof of theorem~\ref{book_new-thm-5-2-2}. Then we should replace $\textup{(\ref{book_new-5-2-16})}^*$ by
\begin{equation}
|V|\ge \epsilon _0\rho^2 \qquad \forall x \in B(0, \gamma)
\tag*{$\textup{(\ref*{5-2-26})}^*$}%
\end{equation}
and therefore we need to pick up the scaling functions
\begin{equation}
\gamma =\epsilon |V| + \bar{\gamma}, \qquad
\rho =\gamma^{\frac{1}{2}}\qquad\text{with\ \ } \bar{\gamma}=\frac{1}{2}h^{\frac{2}{3}}\ \ \text{and\ \ }\bar{\rho}=h^{\frac{1}{3}};
\label{5-2-28}
\end{equation}
we select the $\bar{\gamma}$ equal $h^{\frac{2}{3}}$ to keep $\rho\gamma\ge h$. Obviously $|\nabla _x\gamma |\le {\frac{1}{2}}$.

We immediately arrive to the estimate
\begin{multline}
|\R^\W_{x, \varphi, L}|\le
Ch^{1-d}\rho^{d-1}\gamma^{-1}\vartheta\bigl(\frac{h\rho}{L\gamma}\bigr)\asymp\\
\left\{\begin{aligned}
&Ch^{1-d}|V|^{(d-3)/2}\vartheta\bigl(\frac{h}{L}|V|^{-\frac{1}{2}}\bigr)\qquad &&\text{as\ \ }|V|\ge h^{\frac{2}{3}}, \\
&Ch^{-\frac{2}{3}d}\vartheta\bigl(\frac{h^{2/3}}{L}\bigr)
\qquad &&\text{as\ \ }|V|\le h^{\frac{2}{3}}.
\end{aligned}\right.
\label{5-2-29}
\end{multline}
So, we proved
\begin{theorem}\label{thm-5-2-6}
Let conditions \textup{(\ref{5-2-1})}, \textup{(\ref{5-2-3})}, \textup{(\ref{5-2-4})}, \textup{(\ref{5-2-5})} and \textup{(\ref{5-2-6})} be fulfilled. Then estimate \textup{(\ref{5-2-29})} holds.
\end{theorem}

\begin{corollary}\label{cor-5-2-7} Let conditions of theorem~\ref{thm-5-2-6} be fulfilled. Let us assume that
\begin{claim}\label{5-2-30}
$\vartheta (\tau)\tau ^{-s}$ is a monotone decreasing function with some $
s\ge 0$.
\end{claim}
Then as $d+3\ge s$ estimate \textup{(\ref{5-2-27})} holds and as $d+3<s$ estimate
\begin{equation}
|\R^\W_{x, \varphi, L}|\le C_0h^{{\frac{2}{3}}(-d+s)}\vartheta\bigl(\frac{h}{L}\bigr)
\qquad \forall x \in B(0, {\frac{1}{2}})
\label{5-2-31}
\end{equation}
holds.
\end{corollary}

The reader can reformulate theorem~\ref{book_new-thm-5-1-14} for the Schr\"odinger
operator and improved asymptotics with the ``no loop'' condition.
Finally the results of section~4. 1 yield

\begin{theorem}\label{thm-5-2-8}
Let conditions \textup{(\ref{5-2-1})}, \textup{(\ref{5-2-3})}, \textup{(\ref{5-2-4})}, \textup{(\ref{5-2-5})} and \textup{(\ref{5-2-6})} be fulfilled. Then for $\tau
\le V_*= \inf_{B(0, 1)} V$ the following estimate holds
\begin{equation}
|e(x, y, \tau )|\le C'h^{-d}(1 + {\frac{|V_* -\tau|}{h}})^{-l}\quad
\forall x, y \in B(0, 1-\epsilon )
\label{5-2-32}
\end{equation}
where $C'=C'(d, c, l, \epsilon )$ and $l, \epsilon >0$ are arbitrary.
\end{theorem}
Here we use that $|\nabla V|\le C|V-V_*|^{\frac{1}{2}}$.

\section{Asymptotics without spatial mollification and short loops}
\label{sect-5-2-1-4}
In this subsubsection we consider very special case $\vartheta=1$ and $d=1, 2$ when estimate $|\R^\W_x|\le Ch^{1-d}$ fails replaced by $|\R^\W_x| \le Ch^{-\frac{2}{3}}$ and $\cR_x \le Ch^{-\frac{4}{3}}$ as $d=1, 2$. More precisely, as $|V|\ge h^{\frac{2}{3}}$ estimates
$|\R^\W_x||\le Ch^{1-d}|V|^{(d-3)/2}$ holds.

Our purpose is to improve these estimates, possibly adding a \emph{correction term\/}\index{term!correction} associated with the short loop. As $|\nabla_x V|$ is our foe rather than our friend here, we assume that
condition 
\begin{equation}
|V|+|\nabla V|\ge \epsilon_0 \qquad \text{in \ } B(0, 1)
\label{5-2-16}
\end{equation}
holds. Later we will get rid off it by scaling.

Our goal is to prove

\begin{theorem}\label{thm-5-2-9}
Let conditions \textup{(\ref{5-2-1})}, \textup{(\ref{5-2-3})}, \textup{(\ref{5-2-4})}, \textup{(\ref{5-2-5})}, \textup{(\ref{5-2-6})} and \textup{(\ref{5-2-16})} be fulfilled. Then
\begin{multline}
|e(x, x, 0)-\kappa_0(x)h^{-d} -
h^{-\frac{2}{3}d} |\nabla V(x)|_g ^{\frac{1}{3}d}
\cQ\bigl(W(x)h^{-\frac{2}{3}} \bigr) g|\le \\
Ch^{1-d}\gamma(x)^{\frac{1}{2}(d-2)}
\label{5-2-33}
\end{multline}
where in the correction term $W(x)\asymp V(x)$ and $\cQ$ will be defined by \textup{(\ref{5-2-59})}, and \textup{(\ref{5-A-11})} respectively and
\begin{equation}
\cQ (\lambda )= O(\lambda^{-\frac{1}{4}(d+3)})\ \ \text{as\ \ }\lambda\to\infty
\label{5-2-34}
\end{equation}
(see \textup{(\ref{5-A-12})} for its asymptotics as $\lambda\to+\infty$) and $\gamma(x)\asymp |W(x)|+h^{\frac{2}{3}}$,
\begin{equation}
|\nabla V|_g =\bigl(\sum _{j, k} g^{jk}\partial_jV\cdot\partial_kV\bigr)^{\frac{1}{2}}.
\label{5-2-35}
\end{equation}
\end{theorem}

The crucial step in the proof is

\begin{proposition}\label{prop-5-2-10}
Let $u(x, y, t)$ be the Schwartz kernel of $e^{ih^{-1}A}$. In frames of theorem~\ref{thm-5-2-9}
\begin{equation}
|F_{t\to h^{-1}\tau}{\bar\chi}_T(t)\Gamma_x u|\le Ch^{1-d}\gamma(x)^{\frac{1}{2}(d-2)}
\qquad \forall \tau: |\tau|\le \epsilon \gamma (x)
\label{5-2-36}
\end{equation}
where $T$ is the small constant.
\end{proposition}

Note that the rescaling arguments yield (\ref{5-2-36}) with $T=\epsilon_0 \gamma(x)^{\frac{1}{2}}$ and this would yield estimates $O(h^{1-d}\gamma^{(d-3)/2})$. Generalizing it to $T=\epsilon_0$ leads to remainder estimates $O(h^{1-d}\gamma^{(d-2)/2})$ which are announced in theorem~\ref{thm-5-2-9}, but with Tauberian main part. This Tauberian expression would have two contributors: one equal $h^{-d}\kappa_0(x)$ which is the main term from $t=0$ and another equal to the correction term which from the loop and also collects all other terms from $t=0$.

\begin{proof}[Proof of proposition \ref{prop-5-2-10}]
Consider first the left-hand expression of (\ref{5-2-36}) with ${\bar\chi}_T(t)$ replaced with $\chi_T(t)$ supported in $\frac{1}{2}T\le |t|\le T$ with $T=T_0$ which is a small constant and with $\gamma(x)\le \epsilon_1 T$; then it would be less than $Ch^s$.

Rescaling $x\mapsto x_\new= x T^{-2}$, $t\mapsto t_\new =tT^{-1}$, $h \mapsto h_\new= h T^{-3}$ and multiplying operator by $T^{-2}$ we reduce the general case of $T\ge \gamma (x)^{\frac{1}{2}}$ to the previous one; so now the left-hand expression of (\ref{5-2-36}) with $\chi_T(t)$ instead of ${\bar\chi}_T(t)$ does not exceed
\begin{equation*}
CT^{-2d} \bigl(\frac{h}{T^3}\bigr)^s
\end{equation*}
where $T^{1-2d}= T\times T^{-2d}$ and the first factor $T$ appears from Fourier transform while $T^{-2d}$ appears because we have a density.

Summation with respect to $T\in [C_0\gamma^{\frac{1}{2}}, T_0]$ results in the same expression as $T=\gamma^{\frac{1}{2}}$ i. e.
\begin{equation*}
C\gamma^{-\frac{1}{2}-d} \bigl(\frac{h}{\gamma^{\frac{3}{2}}}\bigr)^s
\end{equation*}
which does not exceed the right-hand expression of (\ref{5-2-36}).

Therefore we need to prove (\ref{5-2-36}) with $T=C_0\gamma^{\frac{1}{2}}$. Rescaling $x\mapsto x_\new=x\gamma^{-1}$, $t\mapsto t_\new =t\gamma^{-\frac{1}{2}}$, $h \mapsto h_\new= h \gamma^{-\frac{3}{2}}$ and multiplying operator by $\gamma^{-1}$ we reduce (\ref{5-2-36}) to the case $\gamma\asymp 1$.

However, if originally $\gamma (x)\ge C_0h^ {\frac{2}{3}}$ then condition (\ref{5-2-26}) is fulfilled after rescaling this estimate follows from (\ref{book_new-4-1-82}) of \cite{futurebook}. On the other hand, if originally $\gamma \asymp h^ {\frac{2}{3}}$ then after rescaling $h_\new \asymp 1$ and this estimate holds as well.
\end{proof}

Therefore due to Tauberian theorem we arrive to
\begin{corollary}\label{cor-5-2-11} In frames of theorem~\ref{thm-5-2-9}
\begin{equation}
|\R^\T_{x, \varphi, L}| \le Ch^{1-d}\gamma(x)^{\frac{1}{2}(d-2)}\vartheta \bigl(\frac{h}{L}\bigr)
\label{5-2-37}
\end{equation}
as $\N^\T_{x, \varphi, L}$ is defined with $T=T_0$ which is a small constant.
\end{corollary}

Now our goal is to calculate $\N^\T_x$ with the indicated error.
Due to the same arguments as in the proof of proposition~\ref{prop-5-2-10}
\begin{claim}\label{5-2-38}
Estimate (\ref{5-2-37}) remains true as $\N^\T_{x, \varphi, L}$ is defined with
$T=\max\bigl( C\gamma^{\frac{1}{2}}, h^{\frac{1}{3}-\delta}\bigr)$.
\end{claim}

Now let us calculate. We will do first the general calculations in the case of $\gamma(x)\asymp 1$ and then we rescale. To do this we need to prove

\begin{theorem}\label{thm-5-2-12}
Let $\bar{x}$, $\bar{y}$ be fixed points and let $t_1<t_2$ be of the same sign and $t_1\asymp t_2 \asymp (t_2-t_1)\asymp 1$. Assume that
\begin{claim}\label{5-2-39}
There exist only one Hamiltonian trajectory $(x(t), \xi(t))$ such that $x(0)=\bar{y}$, $x(t)=\bar{x}$ and $t\in [t_1, t_2]$; let it happen as $t=\bar{t}$, $\xi(0)=\bar{\eta}$ and $(\bar{t}-t_1)\asymp (t_2-\bar{t})\asymp 1$,
\end{claim}
\begin{equation}
\frac{dx}{dt}\Bigr|_{t=\bar{t}} \asymp \frac{dx}{dt}\Bigr|_{t=\bar{t}}\asymp 1
\label{5-2-40}
\end{equation}
and
\begin{claim}\label{5-2-41}
Map $\Sigma_0\cap T^*_{\bar{y}} X \ni \eta \to \uppi_x \Psi_t(y, \eta)$ is nondegenerate in $(\bar{y}, \bar{\eta})$.
\end{claim}
Then as $t_1\le \bar{t}-\epsilon$, $t_2\ge \bar{t}+\epsilon$, $x=\bar{x}$, $y=\bar{y}$
\begin{equation}
h^{-1}\int_{-\infty}^0 \Bigl(F_{t\to h^{-1}\tau} {\bar\chi}_\epsilon (t-T^*) u \Bigr)\, d\tau \equiv\\
e^{ih^{-1}\phi (x, y)}
\sum_{n\ge 0} b_n(x, y) h^{\frac{1-d}{2}+n}
\label{5-2-42}
\end{equation}
where
\begin{equation}
\phi (x, y) = \int_0^{\bar{t}}\ell (x(t), \xi(t))\, dt, \qquad \ell(x, \xi)\Def \langle \partial_\xi a, \xi\rangle .
\label{5-2-43}
\end{equation}
\end{theorem}

\begin{proof}
First, let us rewrite the left-hand expression of (\ref{5-2-42}) as
\begin{equation}
i \Bigl( F_{t\to h^{-1}\tau} t^{-1} {\bar\chi}_\epsilon (t-\bar{t}) \Bigr)\Bigr|_{\tau=0}.
\label{5-2-44}
\end{equation}
On the other hand we know that
\begin{equation}
u(x, y, t) \equiv (2\pi h)^{-\frac{1}{2} (d+m)} \int e^{ih^{-1}\varphi (x, y, \theta, t)}
\sum_{n\ge 0}b'_n(x, y, t, \theta) \, d\theta
\label{5-2-45}
\end{equation}
where $\theta$ is $m$-dimensional variable and $\varphi (x, y, t, \vartheta)$ is defined in the corresponding way. For example, one can take $m=d$, $\theta=\eta$,
\begin{equation}
\varphi (x, y, t, \eta)= -\langle y, \eta\rangle + \psi (x, t, \eta)
\label{5-2-46}
\end{equation}
where
\begin{equation}
\partial _t \psi = - a(x, \partial_x \psi), \qquad
\psi (x, 0, \eta)=\langle x, \eta\rangle.
\label{5-2-47}
\end{equation}
Let us apply stationary phase method with respect to $\theta, t$; condition (\ref{5-2-39}) and $\partial_t\psi =-a(x, \xi)$ imply that there is only one stationary point and it is $(\bar{\eta}, \bar{t})$. Further, conditions (\ref{5-2-40}) and (\ref{5-2-41}) imply that this is non-degenerate point. Thus we gain a factor $h^{\frac{1}{2}}(d+1)$ and (\ref{5-2-42})--(\ref{5-2-43}) are proven.
\end{proof}

\medskip\noindent
\emph{Proof of theorem~\ref{thm-5-2-9}\/}
To apply theorem~\ref{thm-5-2-12} to our case we need just rescale $x\mapsto x\gamma^{-1}$, $h\mapsto\hbar =h\gamma^{-\frac{3}{2}}$; then as $\hbar\le h^\delta$ i. e.
\begin{equation}
 \gamma \ge h^{\frac{2}{3}(1-\delta)}
\label{5-2-48}
\end{equation}
we conclude that
\begin{equation}
\N_{x, \corr}\Def \N^\T _x-\N^\W _x \equiv
e^{i\phi (x, x)}\sum _{n\ge 0} b_n(x, x)\gamma^{\frac{1}{4}(-d-3-6n)}h^{\frac{1-d}{2}+n}
\label{5-2-49}
\end{equation}
where we also multiplied by $\gamma^{-d}$ since we are dealing with densities.

Thus
\begin{claim}\label{5-2-50}
Correction term $\N_{x, \corr}$ is of magnitude $h^{\frac{1}{2}(1-d)}\gamma^{-\frac{1}{4}(d+3)}$.
\end{claim}
This implies a drastic difference between $d=1$ when the correction term is below remainder estimate $h^{1-d}\gamma(x)^{\frac{1}{2}(d-2)}$ of (\ref{5-2-37}) only as $\gamma \asymp 1$ and $d\ge 2$ it is always so as
\begin{equation}
\gamma \ge \bar{\gamma}_1\Def h^{\frac{2d-2}{3d-1}}
\label{5-2-51}
\end{equation}

\begin{problem}\label{prob-5-2-13}
Consider averaged with respect to spectral parameter correction term and to prove that it is of magnitude
\begin{equation*}
h^{\frac{1}{2}(1-d)}\gamma^{-\frac{1}{4}(d+3)}\vartheta \bigl(\frac{h}{L\gamma^{\frac{1}{2}}}\bigr)
\le h^{\frac{1}{2}(1-d)}\gamma^{-\frac{1}{4}(d+3+2s)}\vartheta \bigl(\frac{h}{L}\bigr)
\end{equation*}
under condition 
\begin{claim}\label{5-1-97}
$\vartheta (\tau)\tau ^{-s}$ is a monotone increasing function with some $s>0$,
\end{claim}

I believe that the complete proof of this statement is worth to be published.
\end{problem}

Let us introduce
\begin{equation}
X^-=\{x:\ V(x)<0\}, \qquad X^0=\partial X^-=\{x:\ V(x)=0\},
\label{5-2-52}
\end{equation}
Without any loss of the generality one can assume that $X^-=\{x_1<0\}$.

\subsection{Case $d=1$. \/}
We can assume without any loss of the generality that
$a(x, \xi)=\beta(x)\bigl( \xi_1^2 -V_0(x)\bigr)$ (we can always get rid off $V_1$ by gradient transform); then
\begin{equation}
W(x_1)= \Bigl(\frac{3}{2}\int_0^{x_1} V_0(y_1)^{\frac{1}{2}}\, d y_1\Bigr)^{\frac{2}{3}}
\label{5-2-53}
\end{equation}
is a travel time from $x_1\in X^-$ to $X^0$ (on energy level $0$, if we replace $\beta$ by $1$) and one can see easily that then
\begin{claim}\label{5-2-54}
Under assumption (\ref{5-2-16}) $W(x)/V_0(x)$ is a smooth and disjoint from $0$ function on $X^-\cup X^0$.
\end{claim}
We redefine $x_1= W (x)$ and then $a(x, \xi)$ will be in the same form as before but with different $\beta$ and with $W(x_1)=x_1$:
\begin{equation}
a(x, \xi)=\beta(x)\bigl( \xi_1^2 - x_1\bigr)
\label{5-2-55}
\end{equation}
Let us prove that \emph{in calculations one can replace $\beta(x)$ by $\beta(x)=1$ and assume that \/}
\begin{equation}
A=h^2D_1^2 - x_1\qquad \text{on\ \ }\bR.
\label{5-2-56}
\end{equation}
Really, let $\bar{x}$ be a point where calculations are done, while $x$ be a ``running'' point. Without any loss of the generality one can assume that $\beta(\bar{x})=1$. Let us rescale as before. Then we can assume that we are at he point with $\gamma (x)$ but on $[-c, c]$
\begin{equation}
|D_1^j\beta (x)|\le c_j \varepsilon^j
\label{5-2-57}
\end{equation}
where $\varepsilon$ is an original $\gamma (\bar{x})$.

Let us apply theorem~\ref{thm-5-2-12}. Note that phase functions for the original operator and for the model operator coincide identically while amplitudes differ by $O(\varepsilon)$ (where $\varepsilon =\gamma$); so an error is $O(h^{1-d}\varepsilon)$ and scaling back we get an error estimate $Ch^{1-d}\gamma^{\frac{1}{2}(d-3)}\varepsilon= Ch^{1-d}\gamma^{\frac{1}{2}(d-1)}$ which is actually better by factor $\gamma^{\frac{1}{2}}$ than we need. Calculations for model operator (\ref{5-2-56}) are produced in Appendix~\ref{sect-5-A-1}. Theorem~\ref{thm-5-2-9} is proven as $d=1$.

\subsection{Case \texorpdfstring{$d\ge 2$}{d\textge 2}. \/} Again without any loss of the generally one can assume that $V(x)=- k(x)x_1$ with $k(x)>0$ disjoint from $0$.

Let
\begin{equation}
\Theta= \{(x, \xi): V(x)=0\}\cap \Sigma_0=\{(x, \xi):\ x_1=0, \xi_j=V_j(x)\}
\label{5-2-58}
\end{equation}
parametrized by $x'=(x_2, \dots, x_d)$.

Consider Hamiltonian trajectory passing through $(0, x')\in \Theta$. One can select $\xi_1$ as a natural parameter along this trajectory, so we get a $d$-dimensional manifold $\Lambda \subset \Sigma_0$; since $\Sigma_0$ is $(2d-1)$-dimensional, there exists $(d-1)$-dimensional variable $\eta$: $\Lambda= \Sigma_0\cap \{\eta=0\}$.

Assume first that
\begin{claim}\label{5-2-59}
There is no magnetic field, i. e. $V_1=V_2=\dots=V_d=0$.
\end{claim}
Then trajectory passing through $\Theta$ as $t=0$ is symmetric: $x(-t)=x(t)$ and $\xi(-t)=\xi(t)$. Then the set of loops coincides with $\Lambda$ and $x$-projection of the loop is exactly as on the left picture below
\begin{figure}[h!]\centering
\subfloat[Original coordinates]{\includegraphics[width=.35\textwidth]{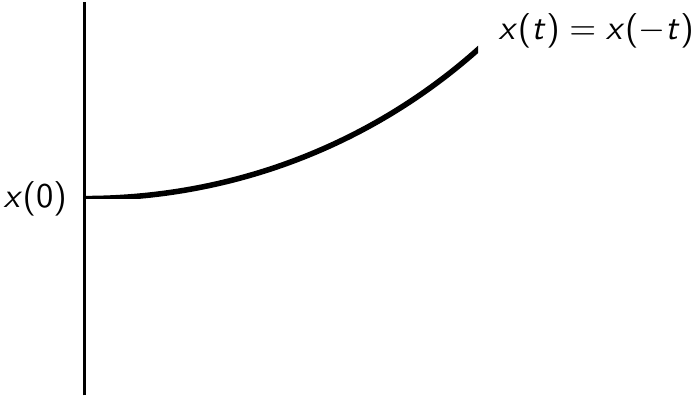}}
\qquad\qquad\subfloat[Straighten coordinates]{
\includegraphics[width=.35\textwidth]{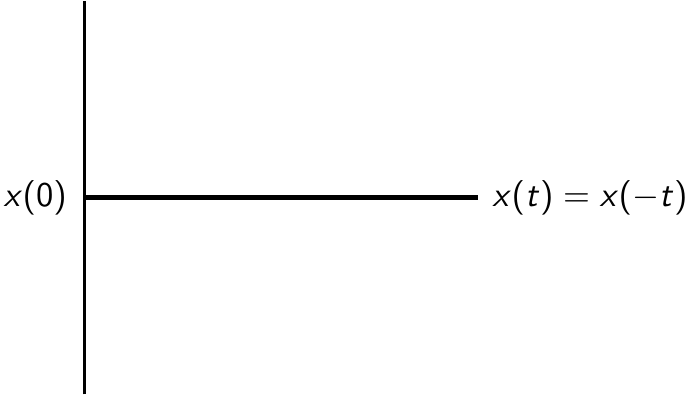}}\caption{\label{shortloops}: Short loops}
\end{figure}

If we introduce new coordinates replacing $x'$ by $x'(0)$ we get picture as on the right. Note that due to the same arguments as for $d=1$ one can assume that
\begin{equation}
a(x, \xi)=\beta(x)\Bigl( \xi_1^2 - x_1 +
\sum_{j, k\ge 2} g^{\prime jk}(x)\bigl(\xi_j-\alpha_j (x)\xi_1\bigr) \bigl(\xi_k-\alpha_k (x)\xi_1\bigr) \Bigr)
\label{5-2-60}
\end{equation}
with positive definite matrix $(g^{\prime jk})$. Then
$\eta =0 $ if and only if $\xi_j-\alpha_j (x)\xi_1=0$ $\forall j=2, \dots, d$.
Note that on $\Sigma_0$ we must have
$\eta=0\implies \{a, \eta\}=0$ i. e.
$\{\xi_1^2 - x_1, \xi_j-\alpha_j (x)\xi_1\}=0$ which easily yields that $\alpha _j(x)=0$ $\forall j=2, \dots, d$ i. e.
\begin{equation}
a(x, \xi)=\beta(x)\Bigl( \xi_1^2 - x_1 + \sum_{j, k\ge 2} g^{\prime jk}(x)\xi_j\xi_k\Bigr)
\label{5-2-61}
\end{equation}
and $\eta=\xi'$.

\begin{wrapfigure}[12]{l}[4pt]{6truecm}
\includegraphics[scale=1]{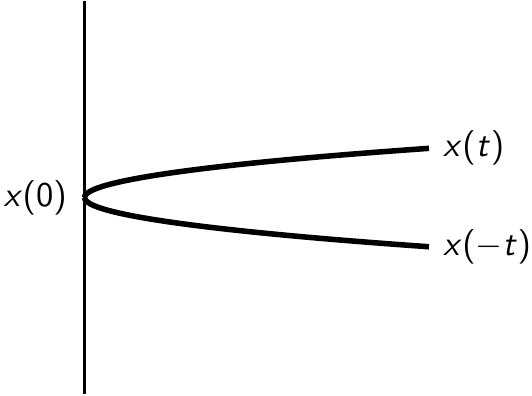}
\caption{\label{fig-loopnomf} Selecting $\xi'\ne 0$ destroys the loop}
\end{wrapfigure}\noindent
One can see easily that in the special case $g^{\prime jk}(x)=\updelta_{jk}$ trajectories are parabolas
\begin{align*}
&\xi_1= t, \quad &&x_1=t^2, \\
&\xi'=\const, \quad &&x'=\const +\xi't \ (j\ge 2)
\end{align*}
and similarly looking in more general case.

Again, as $\gamma \asymp 1$ we can estimate contribution of the loop by $Ch^{1-d}\times h^{\frac{1}{2}(d-1)} = h^{-\frac{1}{2}(d-1)}$ where extra factor $h^{\frac{1}{2}(d-1)}$ due to theorem~\ref{thm-5-2-12}. Scaling procedure leads to estimate (\ref{5-2-50}).

As $g^{\prime jk}=\updelta_{jk}$ Appendix~\ref{sect-5-A-1} implies that estimate~(\ref{5-2-33}) holds. Thus it holds for $g^{\prime jk}=\const$. For variable $g^{\prime jk}$ we get instead of phase $|\xi'|^2$ another phase $|\xi'|^2 + \omega (x, \xi')$ with $\omega = O(\gamma|\xi'|^2)$ which leads to an error with an extra factor $\gamma $ in the error estimate which becomes
$h^{-\frac{1}{2}(d-1)}\gamma^{-\frac{1}{4}(d-1)}$ and is less than the right hand expression of (\ref{5-2-33}).

So, under assumption (\ref{5-2-59}) theorem~\ref{thm-5-2-9} is proven.

We can consider operator (\ref{5-2-61}) and add magnetic field to it.
\begin{wrapfigure}[12]{l}[4pt]{6truecm}
\includegraphics[scale=1]{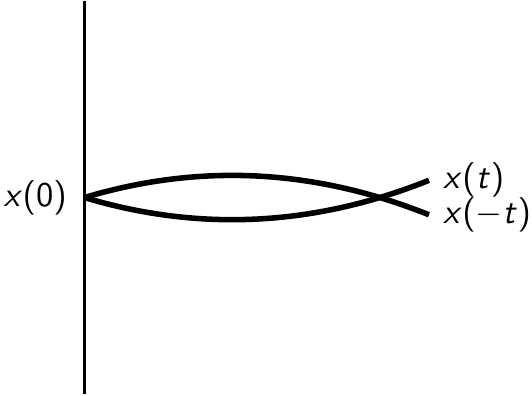}
\caption{\label{fig-loopmf} Short loop with magnetic field}
\end{wrapfigure}\noindent
In this case equalities $\xi(-t)=-\xi(t)$ and $x(-t)=x(t)$ along trajectories passing through $\Theta$ fail, $\Lambda$ loses is value and picture on figure \ref{shortloops}(b) is replaced by the picture on the left where actually the arcs are symmetric only asymptotically. However as before for each $x$ (with $0<x_1<\epsilon$) there exists $\xi=\xi(x)$ such that $(x, \xi)\in \Sigma_0$ and Hamiltonian trajectory passing through $(x, \xi)$ comes back to $x$. So, instead of few looping trajectories and each such trajectory being a loop for each of its points, we have now many looping trajectories but each of them serves only one point.

Further, one can prove that then the phase is changed by $O(\varepsilon)$ where $\varepsilon=\gamma$ and scaling back we get an extra term $O(h^{-1}\gamma^{\frac{3}{2}}\times \varepsilon)=O(h^{-1}\gamma^{\frac{5}{2}})$ in the exponent which leads to the error not exceeding
$Ch^{-\frac{1}{2}(d+1)}\gamma^{-\frac{1}{4}(d-7)}$
which does not exceed the right-hand expression of (\ref{5-2-33}) as $h^{\frac{2}{3}}\le \gamma\le \bar{\gamma}_1$.

Theorem \ref{thm-5-2-9} proven completely. \hfill $\square$

\medskip

\begin{remark}\label{rem-5-2-14}
Obviously
\begin{equation}
W(x) = \const \, \varrho ^{\frac{4}{3}} (x), \qquad \const= \bigl(\frac{3}{2}\bigr)^{\frac{2}{3}}
\label{5-2-62}
\end{equation}
where $\varrho(x)$ is the distance in the metrics $g^{jk}V(x)^{-1}$ from $x$ to $X^0$. Thus one can rewrite the correction term as
\begin{equation}
h^{-\frac{2}{3}d} |\nabla V(x)|^{\frac{d}{3}} \cQ \Bigl( \const \, h^{-\frac{2}{3}}\varrho^{\frac{4}{3}}(x)\Bigr)
\label{5-2-63}
\end{equation}
\end{remark}

Finally, let us get rid off condition $|\nabla V(x)|\ge \epsilon$. To do this let us rewrite first the right-hand expression of (\ref{5-2-33}) as
\begin{equation}
Ch^{-\frac{1}{2}(d-1)}\bigl(\max(|V|, h^{\frac{2}{3}})\bigr)^{\frac{1}{2}(d-2)}
\label{5-2-64}
\end{equation}
thus releasing notation $\gamma$.

Let us introduce a standard scale
$\gamma = \bigl(|V|+|\nabla V|^2\bigr)^{\frac{1}{2}}+ h^{\frac{1}{2}}$ but exclusively at point $\bar{x}$ in question. Then applying the standard rescaling $x\to x\gamma^{-1}$, $h\to \hbar=h\gamma^{-2}$, $V\to V\gamma^{-2}$,
$\nabla V\to \nabla V\gamma^{-1}$ we get instead of (\ref{5-2-64}) expression
\begin{equation*}
Ch^{1-d}\gamma^{d-2}\bigl(\max(|V|\gamma^{-2}, h^{\frac{2}{3}}\gamma^{-\frac{4}{3}})\bigr)^{\frac{1}{2}(d-2)}=
Ch^{1-d)}\bigl(\max(|V|, h^{\frac{2}{3}}\gamma^{\frac{2}{3}})\bigr)^{\frac{1}{2}(d-2)}.
\end{equation*}
As $d\ge 2$ the latter expression can only increase as $\gamma$ is replaced by $1$. However as $d=1$ situation is different and we arrive to
\begin{equation}
C \min (|V|^{-\frac{1}{2}}, h^{-\frac{1}{3}}\gamma^{-\frac{1}{3}})\bigr).
\label{5-2-65}
\end{equation}

Thus we proved

\begin{theorem}\label{thm-5-2-15}
Let conditions \textup{(\ref{5-2-1})}, \textup{(\ref{5-2-3})}, \textup{(\ref{5-2-4})}, \textup{(\ref{5-2-5})} and \textup{(\ref{5-2-6})} be fulfilled. Then

\medskip\noindent
(i) As $d=2$ asymptotics \textup{(\ref{5-2-33})} holds; furthermore one can skip the correction term without penalty as $h^{d-1}|\nabla V|^{d+1}\le |V|^{3d-1}$;

\medskip\noindent
(ii) As $d=1$ the left-hand expression of \textup{(\ref{5-2-33})} does not exceed
\begin{equation}
C\left\{\begin{aligned}
&|V(x)|^{-\frac{1}{2}}\quad &&\text{as}\ \
h^{\frac{2}{3}}|\nabla V(x)|^{\frac{2}{3}}\le | V(x)|\le |\nabla V(x)|^2 \\
&&&\text{or as}\ \ |V(x)|\ge \max\bigl( |\nabla V(x)|^2, h\bigr), \\
&h^{-\frac{1}{3}}|\nabla V(x)|^{-\frac{1}{3}} \quad &&\text{as}\ \
|V(x)|\le h^{\frac{2}{3}}|\nabla V(x)|^{\frac{2}{3}}, \ |\nabla V(x)|\ge h^{\frac{1}{2}}, \\
&h^{-\frac{1}{2}} \quad &&\text{as}\ \ |V(x)|\le h, |\nabla V(x)|\le h^{\frac{1}{2}};
\end{aligned}\right.
\label{5-2-66}
\end{equation}
further, as $|V(x)|\ge \max\bigl( |\nabla V(x)|^2, h\bigr)$ one can skip the correction term without penalty.
\end{theorem}

\section{Spectral kernel calculations for some model operators}
\label{sect-5-A-1}

\subsection{Case $d=1$}
\label{sect-5-A-1-1}
Consider operator (\ref{5-2-56}).

Note first that
\begin{equation}
B=-x_1\implies
\partial_\tau e_B(x, y, \tau)=\updelta (x_1-y_1)\updelta (x_1+\tau).
\label{5-A-1}
\end{equation}
Then making unitary $h$-Fourier transform we conclude that for
\begin{equation}
B=hD_{\xi_1}\implies
\partial_\tau e_B(\xi_1, \eta_1, \tau)=
(2\pi h)^{-1}\exp \bigl(ih^{-1}(\xi_1 -\eta_1)\tau\bigr).
\label{5-A-2}
\end{equation}
Therefore since $\cT^* hD_{\xi_1}\cT= hD_{\xi_1}+\xi_1^2$ for $\cT=\exp (ih^{-1}\frac{1}{3}\xi_1^3)$ we conclude that
\begin{multline}
B=hD_{\xi_1}+\xi_1^2\implies\\
\partial_\tau e_B(\xi_1, \eta_1, \tau)= (2\pi h)^{-1}
\exp \bigl(ih^{-1}\bigl[(\xi_1 -\eta_1)\tau-{\frac{1}{3}}(\xi_1^3-\eta_1^3)\bigr]\bigr).
\label{5-A-3}
\end{multline}
Finally making inverse Fourier transform we conclude that for $A$ defined by (\ref{5-2-56})
\begin{multline}
\partial_\tau e(x_1, y_1, \tau)= \\ (2\pi h)^{-2}\iint
\exp \bigl(ih^{-1}\bigl[(\xi_1 -\eta_1)\tau+x_1\xi_1 -y_1\eta_1 -{\frac{1}{3}}(\xi_1^3-\eta_1^3)\bigr]\bigr)\, d\xi_1\, d\eta_1
\label{5-A-4}
\end{multline}
and therefore
\begin{multline}
\partial_\tau e(x_1, x_1, \tau)= \\ (2\pi h)^{-2}\iint
\exp \bigl(ih^{-1}\bigl[(\xi_1 -\eta_1)(\tau+x_1) -{\frac{1}{3}}(\xi_1^3-\eta_1^3)\bigr]\bigr)\, d\xi_1\, d\eta_1.
\label{5-A-5}
\end{multline}
One can rewrite it as
\begin{align}
&e(x_1, x_1, \tau)= h^{-\frac{2}{3}}\int_{-\infty}^{(x_1+\tau)h^{-\frac{2}{3}}} F(\tau)\, d\tau, \label{5-A-6}\\
\shortintertext{with}
&F(t)= (2\pi )^{-2}\iint \exp \bigl(i\bigl[(\xi_1 -\eta_1)t -{\frac{1}{3}}(\xi_1^3-\eta_1^3)\bigr]\bigr)\, d\xi_1\, d\eta_1=\label{5-A-7}\\
&\qquad\quad 2 (2\pi )^{-2} \iint \exp \bigl(i\bigl[ 2\beta t -{\frac{2}{3}}\beta^3-2\beta\alpha^2\bigr]\bigr)\, d\alpha\, d\beta=\notag\\
&\qquad\quad (2\pi )^{-\frac{3}{2}} \int |\beta|^{-\frac{1}{2}}
\exp \bigl(-i\frac{\pi}{4}\sign\beta +i\bigl[ 2\beta t -{\frac{2}{3}}\beta^3\bigr]\bigr)\, d\beta\notag
\end{align}
where we first substituted $\xi_1=\alpha+\beta$, $\eta_1=\alpha-\beta$ and then integrated with respect to $\alpha$.

The remaining integral has stationary points $\beta =\pm t^{\frac{1}{2}}$ and a singular point $\beta=0$. Decomposing for $t>0$ \ $F(t)=F_1(t)+F_2(t)$,
\begin{align}
&F_1(t)= (2\pi )^{-\frac{3}{2}} \int |\beta|^{-\frac{1}{2}}
\exp \bigl(-i\frac{\pi}{4}\sign \beta +2i\beta t \bigr)\, d\beta= \frac{1}{2}\kappa t^{-{\frac{1}{2}}}\label{5-A-8}\\
\shortintertext{with}
&\kappa = 2(2\pi)^{-\frac{3}{2}}\int_0^\infty \beta^{-\frac{1}{2}}\cos (2\beta -\frac{\pi}{4})\, d\beta, \label{5-A-9}
\end{align}
and
\begin{multline}
F_2(t)=\\(2\pi )^{-\frac{3}{2}} \int |\beta|^{-\frac{1}{2}}\exp\bigl(-i\frac{\pi}{4}\sign\beta +2i\beta t \bigr)
\Bigl(\exp \bigl( -{\frac{2}{3}}i\beta^3\bigr)-1\Bigr)d\beta.
\label{5-A-10}
\end{multline}

Then primitives of $F_1(t)$ and $F_2(t)$ produce exactly main term and correction term in (\ref{5-2-33}):
\begin{multline}
\cQ(t)=\int_\infty^t F_2(t')\, dt'= \\
-2^{-2}(2\pi)^{-\frac{3}{2}}\int |\beta|^{-\frac{3}{2}} \bigl(i\frac{\pi}{4}\sign\beta +2i\beta t \bigr)
\Bigl(\exp \bigl( -{\frac{2}{3}}i\beta^3\bigr)-1\Bigr)\, d\beta.
\label{5-A-11}
\end{multline}
In this integral singularity at $\beta=0$ gives a relatively small contribution (one can prove it is $O(t^{-\frac{3}{2}})$ as $t\to \infty$); the main contribution comes from the stationary points $\beta=\pm t^{\frac{1}{2}}$ and
modulo $O(t^{-\frac{3}{2}})$
\begin{multline}
\cQ(t)\equiv
-2^{-1}(2\pi)^{-1}\sum_{\beta=\pm t^{\frac{1}{2}}} |\beta|^{-2} \exp \bigl(i\frac{\pi}{2}\sign\beta +\frac{4}{3}i\beta t \bigr)=\\
(2\pi)^{-1}t^{-1}\sin \bigl(\frac{4}{3}t^{\frac{3}{2}}\bigr).
\label{5-A-12}
\end{multline}

\subsection{Case $d=2$}
\label{sect-5-A-1-2}
Now we consider operator
\begin{equation*}
A= h^2 D_{x_1}^2+h^2 D_{x_2}^2-x_1
\end{equation*}
leading to $B=\xi_1^2+\xi_2^2+h D_{\xi_1}$
and then instead of (\ref{5-A-5}) we have
\begin{multline}
\partial_\tau e(x_1, x_1, \tau)= \\ (2\pi h)^{-3}\iint
\exp \bigl(ih^{-1}\bigl[(\xi_1 -\eta_1)(\tau+x_1-\xi_2^2) -{\frac{1}{3}}(\xi_1^3-\eta_1^3)\bigr]\bigr)\, d\xi_1d\xi_2d\eta_1.
\label{5-A-13}
\end{multline}
and all previous formulae are adjusting accordingly. Then (\ref{5-A-6})--(\ref{5-A-7}) are replaced by
\begin{align}
&e(x_1, x_1, \tau)= h^{-\frac{4}{3}}\int_{-\infty}^{(x_1+\tau)h^{-\frac{2}{3}}} F(\tau)\, d\tau, \label{5-A-14}\\
\shortintertext{with}
&F(t)=
 2(2\pi )^{-3} \iint \exp \bigl(i\bigl[ 2\beta (t +\gamma^2)-{\frac{2}{3}}\beta^3-2\beta\alpha^2\bigr]\bigr)d\alpha d\beta d\gamma=\label{5-A-15}\\
&\qquad\quad (2\pi )^{-2} \int |\beta|^{-1}
\exp \bigl(i\bigl[ 2\beta t -{\frac{2}{3}}\beta^3\bigr]\bigr)\, d\beta\notag
\end{align}
and (\ref{5-A-8})--(\ref{5-A-10}) are replaced by
\begin{align}
&F_1(t)= (2\pi )^{-2} \int |\beta|^{-1}
\exp \bigl(2i\beta t \bigr)\, d\beta= \kappa + \kappa_1 t^{-3}\label{5-A-16}\\
\shortintertext{with}
&\kappa = 2(2\pi)^{-\frac{3}{2}}\int_0^\infty \beta^{-\frac{1}{2}}\cos (2\beta -\frac{\pi}{4})\, d\beta, \label{5-A-17}
\end{align}
where integral was understood in the sense of distributions and
\begin{equation}
F_2(t)=\\(2\pi )^{-2} i^{-1}\int |\beta|^{-1}\exp\bigl(2i\beta t \bigr)
\Bigl(\exp \bigl( -{\frac{2}{3}}i\beta^3\bigr)-1\Bigr)d\beta.
\label{5-A-18}
\end{equation}
Then (\ref{5-A-11}) is replaced by
\begin{multline}
\cQ(t)=\int_\infty^t F_2(t')\, dt'= \\
-2^{-1}(2\pi)^{-2}\int |\beta|^{-2} \sign\beta \exp \bigl(2i\beta t \bigr)
\Bigl(\exp \bigl( -{\frac{2}{3}}i\beta^3\bigr)-1\Bigr)\, d\beta.
\label{5-A-19}
\end{multline}
In this integral singularity at $\beta=0$ gives a small contribution; really, decomposing $\exp \bigl( -{\frac{2}{3}}i\beta^3\bigr)-1$ into powers of $\beta^3$ we get that he leading term is $O(t^{-2})$.
The main contribution comes from the stationary points
$\beta=\pm t^{\frac{1}{2}}$ and modulo $O(t^{-\frac{3}{2}})$
\begin{equation}
\cQ(t)\equiv
(2\pi)^{-1}t^{-\frac{5}{4}}\sin \bigl(\frac{4}{3}t^{\frac{3}{2}}\bigr).
\label{5-A-20}
\end{equation}

\chapter{Spectral Asymptotics near Boundary}
\label{sect-8-1}

In this section we consider boundary layer type term in $e(x,x,\tau)$ appearing near boundary. Let us recall that due to rescaling technique 
\begin{equation}
e(x,x,\tau)= \N_x^\w + O\bigl(h^{1-d}\gamma^{-1}(x)\bigr)
\label{8-1-1}
\end{equation}
where $\gamma (x)=\dist (x,\partial X)$ and this remainder estimate generated after integration  $O(h^{1-d}|\log h|)$ and basically the whole chapter~\ref{book_new-sect-7} of \cite{futurebook} we spent to eliminate logarithmic factor.

Now we would like to consider $e(x,x,\tau)$ without integration with respect to $x$.

\section[Preliminary analysis]{Preliminary analysis}
\label{sect-8-1-1}

\subsection[Discussion]{Discussion}
\label{sect-8-1-1-1}

Let us recall that in the previous chapter we derived asymptotics of 
\begin{gather}
\N\Def \Gamma \bigl(Q_{1x}e(.,.,\tau)\,^t\!Q_{2y}\bigr)\label{8-1-2}\\
\shortintertext{and}
\N'\Def \Gamma'\eth_x\eth_y \bigl(Q_{1x}e(.,.,\tau)\,^t\!Q_{2y}\bigr)
\label{8-1-3}
\end{gather}
with the remainder $O(h^{1-d})$ provided problem $(A,\eth B)$ is microhyperbolic in multidirection $\cT=(\ell', \nu_1,\dots,\nu_M)$; further, we derived  asymptotics of 
\begin{gather}
\N'_x\Def \Gamma_x\eth_x\eth_y \bigl(Q_{1x}e(.,.,\tau)\,^t\!Q_{2y}\bigr)\label{8-1-4}\\
\shortintertext{and}
\N_x\Def \Gamma_x \bigl(Q_{1x}e(.,.,\tau)\,^t\!Q_{2y}\bigr)
\label{8-1-5}
\end{gather}
with the remainder $O(h^{1-d})$ provided $\ell'=(\ell'_\xi,0)$ and for (\ref{8-1-5}) we also need to assume that $\nu_1=\dots=\nu_M$ in the microhyperbolicity condition; this asymptotics had a boundary layer type term 
$h^{1-d}\Upsilon (x',x_1h^{-1})$ with $D^\alpha \Upsilon (x',r) =O(r^{-\infty})$ as $r\to +\infty$.

Let us discuss this microhyperbolicity condition. Consider first the case of the Laplace operator; without any loss of the generality one can assume that $a(x,\xi)=a'(x,\xi')+\xi_1^2$ where $a(x,\xi')$ is a positive definite quadratic form. Then microhyperbolicity means that
\begin{equation}
\langle \ell'_\xi ,\nabla' _\xi a'\rangle + 2\nu _\pm \xi_1 >0
\qquad \text{as \ \ } \xi_1=\pm \bigl(\tau - a(x,\xi')\bigr)^{\frac{1}{2}}
\label{8-1-6}
\end{equation}
and we can always take $\ell'_\xi =\xi'$, $\nu_+=\nu_- =0$ as long as $a'(x,\xi')>0$ (i.e. $\xi'\ne 0$). 

So, with the exception of the case $\xi'=0$ both microhyperbolicity conditions hold. In this exceptional case however $\nabla'_\xi a'(x,\xi')=0$ and the first term in (\ref{8-1-6}) is $0$ no matter what $\ell'_\xi$ we pick up; so (\ref{8-1-6}) boils up to 
\begin{equation*}
 \pm \nu _\pm \bigl(\tau - a(x,\xi')\bigr)^{\frac{1}{2}} >0
\end{equation*}
and for microhyperbolicity we just take $\nu_\pm =\pm 1$ provided $\tau >0$; this is the standard $\xi'$-microhyperbolicity condition (we deliberately take $\ell'=(\ell'_\xi,0)$). 

However it is not the case as we assume that $\nu_+=\nu_-$; then condition is impossible to satisfy. So, the previous section fails to find asymptotics of $e(x,x,\tau)$ but not because of the rays tangent to the boundary but on the contrary, because of the rays orthogonal to it (in the corresponding metrics).

This is a pleasant surprise because then we can find the solution of non-stationary problem in the form of the standard oscillatory integrals as we did in section~\ref{sect-5-2-1-4}. The analogy does not stop here: the exceptional rays hit the boundary orthogonally, reflect and follow the same path as before forming the short loops of the length $2\dist(x,\partial X)$ where distance is measured in the corresponding Riemannian metrics. 

However there is a difference: in the former case loop appeared because trajectory went ``uphill'', lost velocity and rolled back ``downhill' repeating the path while now it reflects without losing velocity.

\subsection{Pilot model}
\label{sect-8-1-1-2}
Consider $X=\bR^+\times \bR^{d-1}$, $a(\xi)=|\xi|^2$ and either Dirichlet or Neumann boundary conditions. Then propagator is defined by
\begin{gather}
U(x,y,t)= U^0(x,y,t)+U^1(x,y,t)\label{8-1-7}\\
\intertext{with free space solution}
U^0(x,y,t)= (2\pi h)^{-d}\int e^{ih^{-1} (|\xi|^2 t +\langle x-y,\xi \rangle)} \,d\xi\label{8-1-8}
\intertext{and reflected wave}
U^1(x,y,t)= \varsigma (2\pi h)^{-d}\int e^{ih^{-1} (|\xi|^2 t +\langle x-\tilde{y},\xi \rangle)} \,d\xi\label{8-1-9}
\end{gather}
where $\tilde{y}=(-y_1,y_2,\dots,y_d)$ and $\varsigma =\mp 1$ for Dirichlet or Neumann boundary conditions respectively.

Since $U^1$ is responsible to the difference $e^1(x,y,\tau)$ between $e(x,y,\tau)$ and its free space counterpart $e^0(x,y,\tau)$ we  manipulate only with it. Taking $y=x$ we get $\langle x-\tilde{x},\xi \rangle= 2x_1\xi_1$ while $F_{t\to h^{-1}\tau}$ replaces $e^{ih^{-1}\tau}$ by $h\updelta \bigl(\tau - |\xi|^2\bigr)$; so
\begin{multline}
F_{t\to h^{-1}\tau} \Gamma_x U^1 = 
\varsigma (2\pi h)^{1-d}\int \updelta \bigl(\tau - |\xi|^2\bigr) 
e^{2ih^{-1}x_1\xi_1} \,d\xi=\\
\varsigma \const \cdot
h^{1-d}  \tau^{\frac{d}{2}-1}\cdot\left\{\begin{aligned}
&\  \int_0^\pi \cos\bigl(2h^{-1}\tau^{\frac{1}{2}}x_1\cos \phi\bigr)\sin^{d-2}\phi \, d\phi\qquad &d\ge 2,\\
& \cos \bigl(2h^{-1}\tau^{\frac{1}{2}}x_1\bigr)\qquad &d=1
\end{aligned}
\right.
\label{8-1-10}
\end{multline}
where we introduced spherical coordinates $(\phi, \sin \phi \zeta)$ on $\bS^{d-1}$, $\phi\in (0,\pi)$, $\zeta\in \bS^{d-2}$.

It  leads to
\begin{multline}
e^1(x,x,1)= h^{-1}\int ^1_{-\infty} \Bigl(F_{t\to h^{-1}\tau} \Gamma_x U^1\Bigr)\,d\tau=
 h^{-d}\Upsilon_\varsigma (h^{-1}x_1)
\label{8-1-11}
\end{multline}
with
\begin{multline}
\Upsilon_\varsigma (r) =\\
\const \cdot \varsigma
\left\{\begin{aligned}
&\int_0^1   z^{d-1} \int_0^\pi \cos\bigl(2z r\cos \phi\bigr)\sin^{d-2}\phi \, d\phi dz
\qquad &d\ge 2,\\
& \int_0^1  \cos \bigl(2zr\bigr)\qquad &d=1
\end{aligned}
\right.
\label{8-1-12}
\end{multline}
where we plugged $\tau=z^2$.

Here we took $\tau=1$ since any other value could be reduced to this one by rescaling $h\mapsto h\tau^{-\frac{1}{2}}$.

One can find the constant so that for Dirichlet/Neumann problem $e^1(x,x,\tau)=\varsigma e^0(x,x,\tau)$ as $x_1=0$.

Stationary phase method implies that
\begin{equation}
\Upsilon (r) = \const \cdot r^{-(d+1)/2} \cos (2r) + O(r^{-(d+3)/2} )\qquad \text{as\ \ } r\to +\infty.
\label{8-1-13}
\end{equation}
Therefore 
\begin{claim}\label{8-1-14}
With an error  not exceeding
\begin{equation}
Ch^{-d}(x_1h^{-1}+1)^{-\frac{1}{2}(d+1)}+
Ch^{1-d}
\label{8-1-15}
\end{equation}
the standard Weyl expression for $e(x,x,\tau)$ holds;
\end{claim}
this error is better than
$O\bigl(h^{-d}(x_1h^{-1}+1\bigr)^{-1})$ which was due to rescaling but it is only because critical points of $a(\xi')=|\xi'|^2$ are non-degenerate.

\section{Schr\"odinger operator}
\label{sect-8-1-2}
Let us consider Schr\"odinger operator with the symbol
\begin{equation} 
a(x, \xi )=\sum_{j, k} g^{jk}(\xi_j-V_j)(\xi_k-V_k)+V
\label{8-1-16}
\end{equation}
assuming that $\tau -V \asymp 1$ and positive. We want to study trajectories on energy level $\tau$ only and we can assume without any loss of the generality that 
\begin{equation}
\tau=0, \qquad V(x)\le -\epsilon, \qquad V(x)=-1
\label{8-1-17}
\end{equation}
where the last assumption is achieved by division $a(x,\xi)$ by $-V(x)$ which does not affect trajectories on energy level $0$.

\section{General theory}
\label{sect-8-1-2-1}
We can assume without any loss of the generality that $V_1=0$. Let us introduce coordinates $x_1=\dist (x,\partial X)$ in the metrics $(g^{jk})$~\footnote{\label{foot-8-1} I.e. in the final run in $(g^{jk}(-V(x))^{-1})$.} and $x'$ which are constant along trajectories of $a(x,\theta)$ on level $0$ which are orthogonal to the boundary, i.e. such that
\begin{equation}
g^{j1}= \updelta_{j1}.
\label{8-1-18}
\end{equation}
We are interested in the construction of parametrix in
\begin{equation}
\cU=\{|\xi_j-V_j|\le \epsilon\ \qquad j=2,\dots,d\}.
\label{8-1-19}
\end{equation}
It is be much simpler to consider this problem with "time" $x_1$ and with $t$ as one of spatial variables. Then as $l=0$
\begin{equation}
u^l(x,y,t) \equiv 
\int \sum_{\varsigma=\pm}
e^{ih^{-1}(\phi^0_\varsigma (x,y,\eta)+t\eta_0)}\sum_{n\ge 0}
f^l_{\varsigma n}(x,y,\eta)h^{n-d}\,d\eta ,
\label{8-1-20}
\end{equation}
where $\eta=(\eta',\eta_0)=(\eta_2,\dots,\eta_d,\eta_0)$  and  phase functions $\phi^0_\varsigma$ are defined from
\begin{gather}
\partial_{x_1}\phi^l_\varsigma = \lambda_\varsigma (x,\nabla_{x'} \phi_\varsigma ,\eta_0), \label{8-1-21}\\[3pt]
\phi^0_\varsigma (x,y,\eta,\eta_0)|_{x_1=y_1}=\langle x'-y',\eta'\rangle, \label{8-1-22}\\[3pt]
\lambda_\varsigma (x,\xi',\eta_0)= \varsigma\bigl(\eta_0-a(x,\xi')\bigr)^{\frac{1}{2}}.
\label{8-1-23}
\end{gather}
Here amplitudes $f^0_{\varsigma n}(x,y,\eta)$ are smooth, satisfy transport equations and some initial conditions as $x_1=y_1$ but we do not need to calculate them.

Then we can calculate $u^0$, $hD_1u^0$ as $x_1=0$ and define $u^1(x,y,t)$ in the same way (\ref{8-1-20}) but with $\phi^1_\varsigma$ satisfying (\ref{8-1-21}) and
\begin{equation}
\phi^1_\varsigma = - \phi^0_{-\varsigma}\qquad\text{as\ \ } x_1=0.
\label{8-1-24}
\end{equation}
Here amplitudes $f^1_{\varsigma n}(x,y,\eta)$ are smooth, satisfy transport equations and some initial conditions as $x_1=0$ but we do not need to calculate them.

Obviously 
\begin{equation}
\phi^1_\varsigma (y,y,\eta)= y_1 \psi_\varsigma  ( y,\eta)
\qquad\text{as\ \ }x_1=y_1,\label{8-1-25}
\end{equation}
\begin{equation}
\psi_\varsigma  (y,\eta)=\lambda_\varsigma (0,y',\eta) - 
\lambda_{-\varsigma} (0,y',\eta )+O(y_1).
\label{8-1-26}
\end{equation}
Therefore (as $|t|\le T$)
\begin{multline}
\Gamma_y \bigl(Q_{1x}u^1(x,y,t)\,^tQ_{2y}\bigr) \equiv \\
\int \sum_{\varsigma=\pm}
e^{ih^{-1}(y_1\psi_\varsigma (y,\eta)+t\eta_0)}\sum_{n\ge 0}
f^2_{\varsigma n}(y,\eta)h^{n-d}\,d\eta ,
\label{8-1-27}
\end{multline}
where $Q_j=Q(x,hD',hD_t)$ are cut-offs with respect to $\xi',\tau$ and finally
\begin{multline}
F_{t\to h^{-1}\tau}\bar{\chi}_T(t)
\Bigl(\Gamma_y \bigl(Q_{1x}u^1(x,y,t)\,^tQ_{2y}\bigr)\Bigr) \equiv \\
\sum_{\varsigma=\pm}
\int e^{ih^{-1}y_1\psi_\varsigma (y,\eta',\tau)}\sum_{n\ge 0}
f^3_{\varsigma n}(y,\eta',\tau)h^{1+n-d}\,d\eta' .
\label{8-1-28}
\end{multline}
Obviously, representation (\ref{8-1-28}) implies that
\begin{claim}\label{8-1-29}
For small constant $T$ expression (\ref{8-1-28}) does not exceed $Ch^{1-d}$.
\end{claim}
Since (\ref{8-1-29}) is true for $u^0$ under $\xi$-microhyperbolicity condition
\begin{equation}
|\tau -V(y)|\ge \epsilon_0
\label{8-1-30}
\end{equation}
due to free space results, this statement is true also for $u$ and therefore due to Tauberian arguments 

\begin{proposition}\label{prop-8-1-1}
Let \textup{(\ref{8-1-30})} be fulfilled and   $Q_1=Q_1(x,hD')$, $Q_2(x,hD')$ be operators with symbols supported in $\cU$. Then 
$\Gamma_y \bigl(Q_{1x}e^1(x,y,\tau)\,^tQ_{2y}\bigr)$ equals modulo $O(h^{1-d})$ Tauberian expression 
\begin{equation}
h^{-1}\int_{-\infty}^\tau \Bigl(F_{t\to h^{-1}\tau'}\bar{\chi}_T(t)
\bigl(\Gamma_y \bigl(Q_{1x}u^1(x,y,t)\,^tQ_{2y}\bigr)\bigr)\Bigr) \,d\tau'.
\label{8-1-31}
\end{equation}
\end{proposition}

Due to results of subsection \ref{book_new-sect-7-3-1} of \cite{futurebook} we need to calculate only (\ref{8-1-31}) with an extra factor $\chi (\tau-\tau')$ with $\chi$ supported in the small vicinity of $0$:
\begin{equation}
\int_{-\infty}^\tau \chi(\tau-\tau')
\sum_{\varsigma=\pm}
\int e^{ih^{-1}y_1\psi_\varsigma (y,\eta',\tau')}\sum_{n\ge 0}
f^3_{\varsigma n}(y,\eta',\tau')h^{n-d}\,d\eta' d\tau'.
\label{8-1-32}
\end{equation}
Obviously modulo $O(h^{1-d})$ we can skip in the right-hand expression of (\ref{8-1-32}) all the terms with $n\ge 1$; so we get $h^{-d}J$ with obviously defined $J$:
\begin{equation}
J\Def \sum_{\varsigma=\pm}
\int_{-\infty}^\tau \chi(\tau-\tau')
\int e^{i\hbar^{-1}\psi_\varsigma (y,\eta',\tau')} 
f^3_{\varsigma 0}(y,\eta',\tau')\,d\eta' d\tau' 
\label{8-1-33}
\end{equation}
with $\hbar=hy_1^{-1}$.

Further, for $y_1\ge h^{-1}$ one can consider $J$  as an oscillatory integral with a semiclassical parameter $\hbar$. One can see easily that due to (\ref{8-1-25})--(\ref{8-1-26}) in this oscillatory integral the phase function
$\psi_\varsigma (y,\eta,\tau')$ satisfies 
$|\partial_{\tau'}\psi_\varsigma (y,\eta',\tau')|\ge \epsilon_1$ and has non-degenerate critical points with respect to $\eta'$; we denote these points  by $\eta'_\varsigma(y,\tau')$. Therefore 
\begin{equation}
J_\hbar\sim \sum_{\varsigma=\pm}\sum_{k\ge 0}
e^{i\hbar^{-1}\varphi_\varsigma (y,\tau)} 
f^4_{\varsigma 0}(y,\tau) \hbar^{\frac{1}{2}(1+d)+k}
\qquad \text{as\ \ }\hbar \ll 1
\label{8-1-34}
\end{equation}
where
\begin{equation}
\varphi _\varsigma (y,\tau)= \psi_\varsigma (y,\eta'(y,\tau),\tau).
\label{8-1-35}
\end{equation}
Then 
\begin{claim}\label{8-1-36}
$J_\hbar=J^0_\hbar + O(h)$ where $J^0_\hbar$ is also given by (\ref{8-1-33}) but with $y$ replaced by $(0,y')$.
\end{claim}
 Also the fact that $f^l_0$ are defined through transport equations along trajectories implies then
 \begin{proposition}\label{prop-8-1-2}
 Modulo $O(h^{1-d})$ term \textup{(\ref{8-1-31})} coincides with the same term but calculated for $(\bar{A},\eth\bar{B})=(a(0,y',hD_x),\eth b(0,y',hD_x))$.
 \end{proposition}
 
\subsection{Calculations and main theorem}
\label{sect-8-1-2-2}

Now in virtue of propositions~\ref{prop-8-1-1},~\ref{prop-8-1-2} we need only to recalculate $\Upsilon(r)$ in the case of more general boundary conditions than Dirichlet and Neumann, and then make change of variables.

Repeating arguments of subsection~\ref{sect-8-1-1-2} one can prove easily that
\begin{proposition}\label{prop-8-1-3}
Consider Schr\"odinger operator with $g^{jk}=\updelta_{jk}$, $V=-1$, $\tau=0$ and the boundary condition
\begin{equation}
\bigl(B_0(x,hD') hD_1+iB_1(x,hD')\bigr)u|_{x_1=0}=0
\label{8-1-37}
\end{equation}
with
\begin{equation}
|b_0(x',\xi')+|b_1(x',\xi')|\ge \epsilon.
\label{8-1-38}
\end{equation}
Assume that $\supp q_j \subset \cU$ defined by \textup{(\ref{8-1-19})}.
Then

\medskip\noindent
(i) The boundary layer term is  $h^{-d}\Upsilon (h^{-1}x_1)$ with
\begin{multline}
\Upsilon (r) = \\
(2\pi)^{-d}\int \Bigl(\int_{\lambda_-(\xi',\tau)}^{\lambda_+ (\xi',\tau)} \bigl(\xi_1+ i\beta(\xi')\bigr) \bigl(\xi_1-i\beta(\xi')\bigr)^{-1}
e^{2ir \xi_1}\,d\xi_1\Bigr)q_1(\xi')q_2(\xi')\,d\xi'
\label{8-1-39}
\end{multline}
where $b_j$ are principal symbols of $B_j$ and due to self-adjointness assumption $\beta=-ib_1:b_0$ is real-valued. 

\medskip\noindent
(ii) Here function $\Upsilon(r)$ admits decomposition
\begin{equation}
\Upsilon (r)\sim \sum_{\varsigma=\pm} e^{2ir\varsigma} \sum_{n\ge 0}\varkappa_{\varsigma n}r^{-\frac{1}{2}(d+1)-n}\qquad\text{as\ \ }
r\to +\infty.
\label{8-1-40}
\end{equation}
\end{proposition}

Therefore changing variables and using asymptotics in the zone where theorems~\ref{book_new-thm-7-3-10} and \ref{book_new-thm-7-3-11} of \cite{futurebook} work we arrive to

\begin{theorem}\label{thm-8-1-4}
Let us consider Schr\"odinger operator with self-adjoint boundary condition \textup{(\ref{8-1-37})} and let condition $V(x)\le -\epsilon $, 
be fulfilled. Further, let ellipticity condition
\begin{equation}
|b_0(x',\xi')\lambda (x',\xi')+b_1(x',\xi')|\ne 0 \qquad
\text{on\ \ }\supp q_1\cup \supp q_2
\label{8-1-41}
\end{equation}
be fulfilled where $\lambda (x',\xi')$ is a root of $a(x,\xi_1,\xi')=0$ with 
$\Im \lambda >0$~\footnote{\label{foot-8-2} So, this is the condition only in the complement of $\uppi \Sigma|_{\partial X}$ where $\Sigma=\{(x,\xi)\in T^*X, a(x,\xi)=0\}$.}.

Then  
\begin{equation}
\Gamma_x \bigl(Q_{1x}e(x,y,0)\,^tQ_{2y}\bigr) - \kappa_0(x)h^{-d}
-h^{-d}\Upsilon (x',h^{-1}x_1)\sqrt{g}|\le Ch^{1-d}
\label{8-1-42}
\end{equation}
with Weyl coefficient $\kappa_0$ and with $g=\det (g^{jk})^{-1}$ where $x_1$ is the distance from $x$ to $\partial X$ in the metrics $g^{jk}V^{-1}$ and $\Upsilon(r)$ is defined by  \textup{(\ref{8-1-38})} and satisfies \textup{(\ref{8-1-40})}.
\end{theorem}

\begin{remark}\label{rem-8-1-5}
(i) Due to theorems~\ref{thm-8-1-4} and \ref{thm-8-1-6} the above theorem but with  $\Upsilon(r)$ replaced by 
$\Upsilon (r)+ \Upsilon_b(r)$ remains true if ellipticity condition (\ref{8-1-41}) is replaced by microhyperbolicity condition
\begin{multline}
|b_0(x',\xi')\lambda (x',\xi')+b_1(x',\xi')|+\\
|\nabla_{\xi'} \bigl(b_0(x',\xi')\lambda (x',\xi')+b_1(x',\xi')\bigr)|\ne 0 \qquad
\text{on\ \ }\supp q_1\cup \supp q_2
\label{8-1-43}
\end{multline}
where $\Upsilon_b(r)\in \sS(\bR^+)$;

\medskip\noindent
(ii) Standard Weyl asymptotics holds with the remainder estimate $O(h^{1-d})$ for $e(x,x,\tau)$ as $x_1\ge h^{(d-1)/(d+1)}$.
\end{remark}

\section{Generalizations}
\label{sect-8-1-3}

Now we need to consider the generalizations to the systems. We consider only $Q_j$ with symbols supported in $\cU$. The same approach leads to the following

\begin{theorem}\label{thm-8-1-6}
Let $(A,\eth B)$ be self-adjoint and let $\cU\subset T^*\partial X$.  Assume that
\begin{claim}\label{8-1-44}
As $(x,\xi',\tau)\in [0,\epsilon]\times \cU\times [-\epsilon,\epsilon] $ all the roots $\xi_1$ of equation 
\begin{equation}
\det \bigl(\tau -a(x,\xi)\bigr)=0
\label{8-1-45}
\end{equation}
are $\lambda^\pm _j(x,\xi',\tau)$ with $j=1,\dots, M^\pm$ which are real-valued, smooth and of constant multiplicities and
\begin{equation}
\pm \partial_\tau \lambda^\pm _j >0;
\label{8-1-46}
\end{equation}
\end{claim}
Then

\medskip\noindent
(i) Reflected solution has decomposition
\begin{multline}
F_{t\to h^{-1}\tau}\bar{\chi}_T(t)
\Bigl(\Gamma_x \bigl(Q_{1x}u^1(x,y,t)\,^tQ_{2y}\bigr)\Bigr) \equiv \\
\sum_{1\le j\le M^+, 1\le k\le M^-}
\int e^{ih^{-1}x_1\psi_{jk} (x,\eta',\tau)}\sum_{n\ge 0}
f_{jk n}(x,\eta',\tau)h^{1+n-d}\,d\eta' .
\label{8-1-47}
\end{multline}
with real phase functions $\psi_{jk}$ and does not exceed
$Ch^{1-d}(1+h^{-1}y_1)^{-\frac{1}{2}(r-1)}$;

\medskip\noindent
(ii) Asymptotics \textup{(\ref{8-1-42})} holds with 
$h^{-d}\Upsilon (x,h^{-1}x_1)$ constructed from \textup{(\ref{8-1-47})} in the usual Tauberian way; thus
\begin{multline}
\Upsilon (x,h^{-1}x_1) \equiv \\
\sum_{1\le j\le M^+, 1\le k\le M^-}
\int^\tau \int e^{ih^{-1}x_1\psi_{jk} (y,\eta',\tau')}\sum_{n\ge 0}
f_{jk n}(y,\eta',\tau')\,d\eta'd\tau' ;
\label{8-1-48}
\end{multline}
(iii) Under assumption \begin{phantomequation}\label{8-1-49}\end{phantomequation}
\begin{multline}
\nabla_ \xi' \bigl(\lambda^+_j(x,\xi',\tau)-\lambda^-_k(x,\xi',\tau)\bigr)=0\\
\implies
\rank \Hess_{\xi'}\bigl(\lambda^+_j(x,\xi',\tau)-\lambda^-_k(x,\xi',\tau)\bigr)\ge r
\tag*{$\textup{(\ref*{8-1-49})}_r$}\label{8-1-47-r}
\end{multline}
expressions \textup{(\ref{8-1-47})} and \textup{(\ref{8-1-48})} do not exceed 
$Ch^{1-d}(1+h^{-1}x_1)^{-\frac{1}{2}(r-1)}$ and
$C(1+h^{-1}x_1)^{-\frac{1}{2}(r+1)}$ respectively;

\medskip\noindent
(iv) Under condition $\textup{(\ref{8-1-49})}_3$ one can replace $x$ by $x'$ in the phase functions and amplitudes  in \textup{(\ref{8-1-48})}; then 
$\Upsilon (y',h^{-1}y_1)$ is constructed  for the problem $\bigl(a(0,y',hD_x),\eth b(y',hD_x)\bigr)$;

\medskip\noindent
(v) As $r=d-1$ $h^{-d}\Upsilon(h^{-1}x_1)$ admits stationary phase decompositions 
\begin{equation}
\sum_{j,k} e^{ix_1h^{-1}\mu_{jk}(x',\tau)}\sum_{n} f'_{jkn} (x',\tau)
(x_1h^{-1})^{-\frac{1}{2}(d+1)-n}\quad \mod O(h^{1-d});
\label{8-1-50}
\end{equation}
\end{theorem}

\begin{remark}\label{rem-8-1-7}
(i) An easy proof is left to the reader. 

\medskip\noindent
(ii) In contrast to the Schr\"odinger operator we do not define $x_1$ as the distance to the boundary; eliminating $x_1$ from the phases $\psi_{jk}$ without deteriorating estimate is possible as $r\ge 3$ while eliminating $x_1$ from amplitudes  is possible as $r\ge 1$.
\end{remark}

\input IRO23.bbl

\end{document}

%% file: defines_IRO23.tex
\theoremstyle{plain}
\newtheorem{theorem}{Theorem}[chapter]

\newtheorem{proposition}[theorem]{Proposition}
\newtheorem{corollary}[theorem]{Corollary}

\theoremstyle{definition}

\theoremstyle{remark}
\newtheorem{remark}[theorem]{Remark}

\newtheorem{problem}[theorem]{Problem}

\numberwithin{equation}{chapter}

\DeclareMathAlphabet{\mathpzc}{OT1}{pzc}{m}{it}

 \newcommand{\cR}{\mathcal{R}}
 \newcommand{\cQ}{\mathcal{Q}}
 
 \newcommand{\cT}{\mathcal{T}}
 \newcommand{\cU}{\mathcal{U}}

 \newcommand{\sL}{\mathscr{L}}

 \newcommand{\sS}{\mathscr{S}}

\newcommand{\corr}{{\mathsf{corr}}}

\newcommand{\N}{{\mathsf{N}}}

\newcommand{\R}{{\mathsf{R}}}

\newcommand{\T}{{\mathsf{T}}}

\newcommand{\W}{{{\mathsf{W}}}}
\newcommand{\w}{{\mathsf{w}}}

\newcommand{\const}{{\mathsf{const}}}
\newcommand{\dist}{{{\mathsf{dist}}}}

\newcommand{\new}{{{\mathsf{new}}}}

\newcommand{\bC}{{\mathbb{C}}}
\newcommand{\bH}{{\mathbb{H}}}

\newcommand{\bR}{{\mathbb{R}}}
\newcommand{\bS}{{\mathbb{S}}}

\def\1{\boldsymbol {|}}
\newcommand{\Def}{\mathrel{\mathop:}=}

\newcommand{\Hess}{\operatorname{Hess}}
\renewcommand{\Im}{\operatorname{Im}}       

\newcommand{\rank}{\operatorname{rank}}
\newcommand{\sign}{\operatorname{sign}}

\newcommand{\supp}{\operatorname{supp}}

\newenvironment{claim}[1][{\textup{(\theequation)}}]{\refstepcounter{equation}\vglue10pt
\begin{trivlist}
\item[{\hskip\labelsep#1}]}{\vglue10pt\end{trivlist}}

\newenvironment{claim*}[1][{}]{\vglue10pt
\begin{trivlist}
\item[{\hskip\labelsep#1}]}{\vglue10pt\end{trivlist}}

\newenvironment{phantomequation}[1][]{\refstepcounter{equation}}{}
\newcounter{note}

\DeclareTextCommand{\textinfty}{PU}{\9042\036}

\DeclareTextCommand{\textge}{PU}{\9042\145}
\DeclareTextCommand{\textle}{PU}{\9042\144}
\DeclareTextCommand{\texthat}{PD1}{\136}

%% file: IRO23.bbl
\bibliographystyle{alpha}

\providecommand{\bysame}{\leavevmode\hbox to3em{\hrulefill}\thinspace}

\vglue .06truein

\begin{tabular}{rrl}
&{\hskip 200 pt} &Department of Mathematics,\cr
&&University of Toronto,\cr
&&40, St.George Str.,\cr
&&Toronto, Ontario M5S 2E4\cr
&&Canada\cr
&&ivrii@math.toronto.edu\cr
&&Fax: (416)978-4107\cr
\end{tabular}